\documentclass[reqno]{amsart}
\usepackage{amssymb}
\usepackage{setspace}
\usepackage{hyperref}
\newtheorem{theorem}{Theorem}[section]

\newtheorem{corollary}[theorem]{Corollary}
\newtheorem{thm}[theorem]{Theorem}
\newtheorem{remark}[theorem]{Remark}

\numberwithin{equation}{section}

\begin{document}
\title[]{Explicit, recurrent, determinantal expressions of the $k$th power of formal power series and applications to the generalized Bernoulli numbers}

\author{ Said Zriaa and Mohammed Mou\c{c}ouf}
\address{ Said Zriaa and Mohammed Mou\c{c}ouf,
University Chouaib Doukkali.
Department of Mathematics, Faculty of science
Eljadida, Morocco}
\email{saidzriaa1992@gmail.com}
\email{moucouf@hotmail.com}

\begin{abstract}
In this work, the authors provide closed forms and recurrence expressions for computing the $k$th power of the formal power series, some of them in terms of a determinant of some matrices. As a consequence, we obtain the reciprocal of the unit of any formal power series. We apply these results to the generalized Bernoulli numbers and Bernoulli numbers, we derive new closed-form expressions and some recursive relations of these famous numbers. In addition, we present several identities in determinant form. Using these results, an elegant generalization of a well known identity of Euler is presented. We also note some connections between the Stirling numbers of the second kind and the generalized Bernoulli numbers.
\end{abstract}

\maketitle	
\section{Formal power series}
An $n\times n$ Toeplitz-Hessenberg matrix over a unitary commutative ring is the transpose of the following matrix
\begin{equation*}
TH_{n}(a_{0},a_{1},\cdots,a_{n})=\begin{pmatrix}
a_{1}   & a_{2}        &\cdots  &\cdots   & \cdots  &a_{n}\\
a_{0}       &     \ddots    &\ddots  &         &        &\vdots      \\
0       &    \ddots     & \ddots &  \ddots &        &\vdots\\
\vdots  &    \ddots     & \ddots &\ddots   &  \ddots      &\vdots\\
\vdots  &               & \ddots &\ddots   & \ddots & a_{2}\\
0       & \cdots        & \cdots &   0     & a_{0}      &a_{1}
\end{pmatrix}
\end{equation*}
where we assume $a_{0}\neq 0$. This class of matrices has significance in both pure and applied mathematics. They are encountered in various applications (see, e.g.,\cite{Xwchang,Hchen,Ainselberg,Ajanssen}). The determinant of Toeplitz-Hessenberg matrices has realized significant importance due to their frequent use in several applications. Using elementary properties of determinants, it is clear that for all $n\geq 1$: 
\begin{equation*}
\det(TH_{n})=\sum_{i=1}^{n}(-1)^{i-1}a_{0}^{i-1}a_{i}\det(TH_{n-i}) \,\,\ \mbox{with} \,\,\ \det(TH_{0})=1.
\end{equation*}
In addition to the last recursive formula, there exists an explicit formula for computing $\det(TH_{n})$, known in the literature as Trudi's formula~\cite{Tmuir}, it states that for any positive integer $n$
\begin{equation*}
\det(TH_{n})=\sum_{\substack{k_1+k_2+\cdots+k_n=l \\ k_1+2k_2+\cdots+nk_n=n}}l!\frac{a_{1}^{k_1}a_{2}^{k_2}\cdots a_{n}^{k_n}}{k_1!k_2!\cdots k_n!}(-a_{0})^{n-l} \,\,\ \mbox{with} \,\,\ k_1,k_2,\cdots,k_n\geq 0.
\end{equation*} 
another equivalent formula is
\begin{equation*}
\det(TH_{n})=\sum_{k=1}^{n}(-a_{0})^{n-k}\sum_{\substack{i_1+i_2+\cdots+i_k=n}}a_{i_1}a_{i_2}\cdots a_{i_k} \,\,\ \mbox{with} \,\,\ i_1,i_2,\cdots,i_k\geq 1.
\end{equation*}
Note that when $a_{0}=1$, the formula of Trudi is known as Brioschi's formula~\cite{Tmuir}.
\\
\indent
Several combinatorial identities involving multinomial coefficients can be generated using Trudi's formula. In our results, we may use this identity to produce various new algebraic identities.
\\
\indent
The theory of formal power series still an exciting subject. Therefore, the formal power series is an essential topic in diverse areas of mathematics due to their natural appearance in various applications (see~ \cite{Gevere,Imgessel,Iniven,Jriordan,Dzeitlin} and references therein). They are a powerful tool in the theoretical computer sciences. The formal power series are closely related to number theory and many others topics such that special functions and combinatorics (for more details see~\cite{Hwgould,Iniven,Bsambale} and the references cited therein). The foundations of this important theory were laid by Lagrange, Moivre, L. Euler. For a developed systematic theory and a rigorous foundation of formal power series, the elegant expository work of Niven~\cite{Iniven} is recommended.\\
There are several concrete situations that make the formal power series a useful tool in number theory.
In additive number theory, the representation of some functions by power series is more convenient. For example, the study of number-partitions using the formal power series provides ingenious results. The famous mathematician Euler began the theory of partitions in 1748 and he discovered the ordinary generating function of the well known partition function $p(n)$ by means of the q-series. G. E. Andrews~\cite{Geandrews} stated Euler's result as
\begin{equation*}
1+\sum_{n=1}^{\infty}p(n)q^{n}=\prod_{n=1}^{\infty}(1-q^{n})^{-1}
\end{equation*}
The pentagonal number theorem is one of Euler's most profound results, it relates the product and the power series representations of Euler's function
\begin{equation*}
1+\sum_{n=1}^{\infty}a_{n}q^{n}=\prod_{n=1}^{\infty}(1-q^{n})
\end{equation*}
where
\begin{eqnarray*}
a_{n}=
\left\{
\begin{array}{lll}
1,& \text{if} &n=m(3m+1)/2, \,\ m\in\mathbb{Z} \,\ \text{even},\\
-1,& \text{if} &n=m(3m+1)/2, \,\ m\in\mathbb{Z} \,\ \text{odd},\\
0,& \text{otherwise.}
\end{array}
\right.
\end{eqnarray*}
 
Here the generalized pentagonal numbers are the exponents of nonzero terms.\\
On the other hand, it is clear that
\begin{equation*}
\bigg(1+\sum_{n=1}^{\infty}(-1)^{n}\bigg\{x^{\frac{3n^{2}-n}{2}}+x^{\frac{3n^{2}+n}{2}}\bigg\}\bigg)\bigg(1+\sum_{n=1}^{\infty}p(n)x^{n}\bigg)=1
\end{equation*}
Also by a simple manipulation of formal power series, we can easily obtain that
\begin{equation*}
p(m)=\sum_{n=1}^{\infty}(-1)^{n+1}\bigg\{p\bigg(m-(3n^{2}-n)/2\bigg)+p\bigg(m-(3n^{2}+n)/2\bigg)\bigg\}
\end{equation*}
This identity leads to an efficient method of calculating the partition function $p(n)$. We note that this formula is known in the large literature as Euler's pentagonal number recurrence.\\
It should be mentioned that by calculating the reciprocal of the ordinary generating function of $p(n)$, we can express the function $p(n)$ in determinant formulas. More precisely, using one of our results we have
\begin{equation*}
p(n)=(-1)^{n}
\left|\begin{array}{cccccc}
a_{1}   & a_{2}        &\cdots  &\cdots   & \cdots  &a_{n}\\
 1      &     \ddots    &\ddots  &         &        &\vdots      \\
0       &    \ddots     & \ddots &  \ddots &        &\vdots\\
\vdots  &    \ddots     & \ddots &\ddots   &  \ddots      &\vdots\\
\vdots  &               & \ddots &\ddots   & \ddots & a_{2}\\
0       & \cdots        & \cdots &   0     & 1      &a_{1}
\end{array}\right|.
\end{equation*}
and
\begin{equation*}
a_n=(-1)^{n}
\left|\begin{array}{cccccc}
p(1)   & p(2)        &\cdots  &\cdots   & \cdots  &p(n)\\
 1      &     \ddots    &\ddots  &         &        &\vdots      \\
0       &    \ddots     & \ddots &  \ddots &        &\vdots\\
\vdots  &    \ddots     & \ddots &\ddots   &  \ddots      &\vdots\\
\vdots  &               & \ddots &\ddots   & \ddots & p(2)\\
0       & \cdots        & \cdots &   0     & 1      &p(1)
\end{array}\right|.
\end{equation*}
\\
\indent
Now let $S=a_{0}+a_{1}X+a_{2}X^{2}+\cdots=\sum_{i\geq 0}a_{i}X^{i}$ be any formal power series over a unitary commutative ring $R$. For any nonegative integer $k$, let us write the $k$th power of $S$ as
\begin{equation*}
S^{k}=a_{0}(k)+a_{1}(k)X+a_{2}(k)X^{2}+\cdots=\sum_{i\geq 0}a_{i}(k)X^{i}
\end{equation*}
where the coefficients $a_{m}(k)$ can be calculated recursively by
\begin{eqnarray}\label{prr}
\left\{\begin{array}{lll}
a_{0}(k)&=&a_{0}^{k},\\
a_{m}(k)&=&\displaystyle\frac{1}{ma_{0}}\sum_{i=1}^{m}(ik-m+i)a_{i}a_{m-i}(k),\quad m\geq 1.
\end{array}\right.
\end{eqnarray}
This identity is one of the most elegant and oldest formulas used to calculate the $k$th powers of a given power formal series. It is usually known as the J.C.P. Miller formula, for example, it appears in the well known book of P. Henrici~\cite{Phenrici}, but actually, the history of this intersting identity is very old. In~\cite{Hwgould} H. W. Gould attributed this formula to Euler 
and provides extensive history and though discussion for this beautiful recurrence relation.\\
The relation~\eqref{prr} provides an efficient way to determine the coefficients $a_{m}(k)$. However, this method is valid only in the case where $m$ and $a_{0}$ are units in the ring $R$. The goal here is to propose a new general approach, which requires no assumption, for computing the $k$th power of formal power series.
\\
\indent
Throughout this section, $R$ will be used to denote a commutative unitary ring.\\
Let $k$ and $r$ be two integers and consider the linear map defined by
\begin{displaymath}
[.]_{k}^{r}:
  \begin{array}{rcl}
   R[X]  & \longrightarrow & R[X] \\
    X^{i} & \longmapsto & X^{k-(r-i)}\binom{k}{r-i} \\
  \end{array}
\end{displaymath}
for $k\leq p$, we use the convention that
\begin{equation}
X^{k-p}\binom{k}{p}=\delta_{k,p}
\end{equation}
Then using an important result of~\cite{Moucouf}, we have the following theorem which gives a closed-form expression of the $k$th power of formal power series.
\begin{thm}\label{thm1}
Let $\{a_{n}\}_{n\in \mathbb{N}}$ be a sequence of elements of $R$ such that $a_{0}$ is invertible. For all integer $k$, the $k$th power of the following formal power series 
$$S=a_{0}+a_{1}X+a_{2}X^{2}+\cdots=\sum_{i\geq 0}a_{i}X^{i}$$
are given as follows:
\begin{eqnarray}
S^{k}=[\mathcal{X}_{0}]_{k}^{0}(a_{0})+[\mathcal{X}_{1}]_{k}^{1}(a_{0})X+[\mathcal{X}_{2}]_{k}^{2}(a_{0})X^{2}+\cdots=\sum_{i\geq 0}[\mathcal{X}_{i}]_{k}^{i}(a_{0})X^{i}
\end{eqnarray}
where $\mathcal{X}_{n}$ is the polynomial
\begin{equation}
\delta(-X, a_{1}, \ldots, a_{n})=
\left|\begin{array}{cccccc}
a_{1}   & a_{2}        &\cdots  &\cdots   & \cdots  &a_{n}\\
-X       &     \ddots    &\ddots  &         &        &\vdots      \\
0       &    \ddots     & \ddots &  \ddots &        &\vdots\\
\vdots  &    \ddots     & \ddots &\ddots   &  \ddots      &\vdots\\
\vdots  &               & \ddots &\ddots   & \ddots & a_{2}\\
0       & \cdots        & \cdots &   0     & -X      &a_{1}
\end{array}\right|.
\end{equation}
and $\mathcal{X}_{0}=1$. On the other hand, we have
\begin{equation*}
\delta(-X, a_{1}, \ldots, a_{n})=\sum_{\substack{k_1+k_2+\cdots+k_n=l \\ k_1+2k_2+\cdots+nk_n=n}}l!\frac{a_{1}^{k_1}a_{2}^{k_2}\cdots a_{n}^{k_n}}{k_1!k_2!\cdots k_n!}X^{n-l}
\end{equation*}
more explicitly
\begin{equation*}
[\mathcal{X}_{n}]_{k}^{n}(a_{0})=\sum_{\substack{k_1+k_2+\cdots+k_n=l \\ k_1+2k_2+\cdots+nk_n=n}}k(k-1)\cdots(k-l+1)\frac{a_{1}^{k_1}a_{2}^{k_2}\cdots a_{n}^{k_n}}{k_1!k_2!\cdots k_n!}a_{0}^{k-l}
\end{equation*}
\end{thm}
\begin{proof}
It is easy to see that the canonical mapping
\begin{eqnarray*}
[a_{0}, a_{1}, a_{2}, \ldots] \longrightarrow a_{0}+a_{1}X+a_{2}X^{2}+\cdots
\end{eqnarray*}
is an isomorphism from the ring of infinite semicirculant matrices onto the ring of formal power series. Thus the result follows immediately from~\cite[Theorem 2.5]{Moucouf}.
\end{proof}
\begin{remark}
It should be noted that if $a_{0}$ is not invertible the last theorem remains valid for calculating the nonnegative kth power of the formal power series. In this case, we just consider the following map
\begin{displaymath}
[.]_{k}^{r}:
  \begin{array}{rcl}
   R[X]  & \longrightarrow & R(X) \\
    X^{i} & \longmapsto & X^{k-(r-i)}\binom{k}{r-i} \\
  \end{array}
\end{displaymath}
\end{remark}
As a consequence of Theorem~\eqref{thm1}, we recover a well-known formula attributed in~\cite{Phenrici} to Wronski and in~\cite{Tmuir} to Trudi. This foumus formula can also be found in~\cite[p.347]{Ainselberg} and~\cite[Lemma 2.4]{Fqi}. Due to its importance, there are various different proofs in the literature. The interested reader is referred for example to~\cite{Phenrici,Fqi,Fqi2}. In the following, we provide a new and elegant proof of Wronski’s formula using some important algebraic tools. 
\begin{theorem}\label{Thm f1}
Let $\{a_{n}\}_{n\in \mathbb{N}}$ be a sequence of the elements of $R$ such that $a_{0}$ is invertible. The inverse of the formal power series  $$S=a_{0}+a_{1}X+a_{2}X^{2}+\cdots=\sum_{i\geq 0}a_{i}X^{i}$$ is given by:
\begin{equation*}
S^{-1}=b_{0}+b_{1}X+b_{2}X^{2}+\cdots=\sum_{i\geq 0}b_{i}X^{i}
\end{equation*}
where
\begin{equation}\label{eq:wroski}
b_{n}=\dfrac{(-1)^{n}}{a_{0}^{n+1}}
\left|\begin{array}{cccccc}
a_{1}   & a_{2}        &\cdots  &\cdots   & \cdots  &a_{n}\\
 a_{0}      &     \ddots    &\ddots  &         &        &\vdots      \\
0       &    \ddots     & \ddots &  \ddots &        &\vdots\\
\vdots  &    \ddots     & \ddots &\ddots   &  \ddots      &\vdots\\
\vdots  &               & \ddots &\ddots   & \ddots & a_{2}\\
0       & \cdots        & \cdots &   0     & a_{0}      &a_{1}
\end{array}\right|.
\end{equation}
Moreover $b_{n}$ can be explicitly expressed as

\begin{equation*}
b_{n}=\sum_{\substack{k_1+k_2+\cdots+k_n=p \\ k_1+2k_2+\cdots+nk_n=n}}\binom{p}{k_{1},\ldots,k_{n}}(-1)^{p}a_{0}^{-1-p}a_{1}^{k_{1}}a_{2}^{k_2}\cdots a_{n}^{k_{n}}  
\end{equation*}
\end{theorem}
\begin{proof}
We first prove that for all polynomial $P(X)$, we have 
$$[P(X)]_{-1}^{r}(a_{0})=\dfrac{(-1)^{r}}{a_{0}^{r+1}}P(-a_{0})$$
Now, it is clear that
\begin{equation*}
[X^{i}]_{-1}^{r}=X^{i-(r+1)}(-1)^{r+i}
\end{equation*}

Let $P(X)=\sum_{i=0}^{s}c_{i}X^{i}$, then by linearity
\begin{align*}
[P(X)]_{-1}^{r}(a_{0})&=\sum_{i=0}^{s}c_{i}[X^{i}]_{-1}^{r}(a_{0})\\
             &=\dfrac{(-1)^{r}}{a_{0}^{r+1}}\sum_{i=0}^{s}c_{i}(-a_{0})^{i}\\
             &=\dfrac{(-1)^{r}}{a_{0}^{r+1}}P(-a_{0})
\end{align*}
Consequently, we have
\begin{align*}
b_{n}&=[\mathcal{X}_{n}]_{-1}^{n}(a_{0})\\
             &=\dfrac{(-1)^{n}}{a_{0}^{n+1}}\mathcal{X}_{n}(-a_{0})\\
             &=\dfrac{(-1)^{n}}{a_{0}^{n+1}}\delta(a_{0},a_{1},\ldots,a_{n})
\end{align*}
\end{proof}
A direct consequence of Theorem~\eqref{Thm f1} is the following
\begin{theorem}\label{thm f2}
\begin{equation*}
\Bigg(1-\sum_{n=1}^{k}a_{n}X^{n}\Bigg)^{-1}=1+\sum_{n\geq 1}b_{n}X^{n}
\end{equation*}
where
\begin{equation*}
b_{n}=\sum_{i_1+2i_2+\cdots+ki_k=n}\frac{(i_1+i_2+\cdots+i_k)!}{i_1!i_2!\cdots i_k!}a_{1}^{i_1}a_{2}^{i_2}\cdots a_{k}^{i_k}
\end{equation*}
\end{theorem}
As a consequence, we can provide the following corollary
\begin{corollary}
For every positive integer $k\geq 1$ and every nonnegative integer $n$, we have
\begin{equation*}
\binom{n+k-1}{n}=\sum_{\substack{i_1+i_2+\cdots+i_k=p \\ i_1+2i_2+\cdots+ki_k=n}}(-1)^{n-p}\binom{p}{i_{1},\ldots,i_{k}}
\binom{k}{1}^{i_1}\binom{k}{2}^{i_2}\cdots \binom{k}{k}^{i_k}
\end{equation*}
\end{corollary}
\begin{proof}
It is easy to check that
\begin{equation*}
(1-X)^{-k}=1+\sum_{n\geq 1}\binom{n+k-1}{n}X^{n}
\end{equation*}
On the other hand, we have
\begin{equation*}
(1-X)^{-k}=\bigg(1-\sum_{n=1}^{k}\binom{k}{n}(-1)^{n-1}X^{n}\bigg)^{-1}
\end{equation*}
From this the result follows easily.
\end{proof}
The following theorem provides another way to obtain a closed-form expression of the $k$th power of formal power series.
\begin{thm}
Let $\{a_{n}\}_{n\geq 0}$ be a sequence of elements of $R$. For all nonnegative integers $k$, the $k$th power of the formal power series 
$$S=a_{0}+a_{1}X+a_{2}X^{2}+\cdots=\sum_{i\geq 0}a_{i}X^{i}$$
are given as follows:

\begin{eqnarray*}
S^{k}=\pmb{a}_{0}^{(k)}+\pmb{a}_{1}^{(k)}X+\pmb{a}_{2}^{(k)}X^{2}+\cdots=\sum_{i\geq 0}\pmb{a}_{i}^{(k)}X^{i}
\end{eqnarray*}
where
\begin{eqnarray}\label{eq:double}
\left\{\begin{array}{lll}
\pmb{a}_{0}^{(k)}&=&a_{0}^{k}\binom{k}{0},\\
\\
\pmb{a}_{n}^{(k)}&=&\displaystyle\sum_{p=1}^{n}a_{0}^{k-p}\binom{k}{p}\sum_{\substack{k_1+k_2+\cdots+k_n=p \\ k_1+2k_2+\cdots+nk_n=n}}\frac{p!}{k_{1}!\ldots k_{n}!} a_{1}^{k_{1}}\cdots a_{n}^{k_{n}},\quad n\geq 1.
\end{array}\right.
\end{eqnarray}
\end{thm}
\begin{proof}
Taking into account the canonical isomorphism
\begin{eqnarray*}
[a_{0}, a_{1}, a_{2}, \ldots] \longrightarrow a_{0}+a_{1}X+a_{2}X^{2}+\cdots
\end{eqnarray*}
from the ring of infinite semicirculant matrices onto the ring of formal power series and Theorem $3.1$ of~\cite{Mou}, we easily obtain the result.
\end{proof}
The following theorem provides another elegant simple formula for computing the $k$th power of any formal power series
\begin{thm}\label{thm: gggg} 
Let $\{a_{n}\}_{n\geq 0}$ be a sequence of elements of $R$. For all nonnegative integers $k$, the kth power of the formal power series 
$$S=a_{0}+a_{1}X+a_{2}X^{2}+\cdots=\sum_{i\geq 0}a_{i}X^{i}$$
are given as follows:
\begin{eqnarray*}
S^{k}=\pmb{a}_{0}^{(k)}+\pmb{a}_{1}^{(k)}X+\pmb{a}_{2}^{(k)}X^{2}+\cdots=\sum_{i\geq 0}\pmb{a}_{i}^{(k)}X^{i}
\end{eqnarray*}
where the sequence $\{\pmb{a}_{n}^{(k)}\}_{n\geq0}$ is determined by:
\begin{equation*}
\pmb{a}_{n+1}^{(k)}=a_{n+1}\pmb{\widehat{a}}_{0}^{(k)}+\cdots+a_{i+1}\pmb{\widehat{a}}_{n-i}^{(k)}+
\cdots+a_{1}\pmb{\widehat{a}}_{n}^{(k)}
\end{equation*}
and
\begin{eqnarray}\label{eq:hhhh}
\left\{\begin{array}{lll}
\pmb{\widehat{a}}_{0}^{(k)}&=&a_{0}^{k-1}\binom{k}{1},\\
\\
\pmb{\widehat{a}}_{n}^{(k)}&=&\displaystyle\sum_{p=1}^{n}a_{0}^{k-p-1}\binom{k}{p+1}\sum_{\substack{k_1+k_2+\cdots+k_n=p \\ k_1+2k_2+\cdots+nk_n=n}}\frac{p!}{k_{1}!\ldots k_{n}!} a_{1}^{k_{1}}\cdots a_{n}^{k_{n}},\quad n\geq 1.
\end{array}\right.
\end{eqnarray}
\end{thm}
\begin{proof}
By considering the canonical isomorphism between the ring of infinite semicirculant matrices and the ring of formal power series and equations $(13)$ and $(16)$ of~\cite{Mou}, we easily obtain the desired result.
\end{proof}
We continue to provide important result for computing the $k$th power of any formal power series. In the following, we present a new recursive relation.
\begin{theorem}\label{thm:bern}
Let $\{a_{n}\}_{n\geq 0}$ be a sequence of elements of $R$. For all nonnegative integers $k$, the $k$th power of the formal power series 
$$S=a_{0}+a_{1}X+a_{2}X^{2}+\cdots=\sum_{i\geq 0}a_{i}X^{i}$$
are given as follows:
\begin{eqnarray*}
S^{k}=\pmb{a}_{0}^{(k)}+\pmb{a}_{1}^{(k)}X+\pmb{a}_{2}^{(k)}X^{2}+\cdots=\sum_{i\geq 0}\pmb{a}_{i}^{(k)}X^{i}
\end{eqnarray*}
where the sequence $\{\pmb{a}_{n}^{(k)}\}_{n\geq0}$ is determined recursively by:
\begin{equation*}
\pmb{a}_{n+1}^{(k)}=a_{n+1}a_{0}^{k-1}\binom{k}{1}+\sum_{i=1}^{n}a_{i}\sum_{l=1}^{k-1}a_{0}^{l-1}\pmb{a}_{n+1-i}^{(k-l)}
\end{equation*}
\end{theorem}
\begin{proof}
We have by Theorem\eqref{thm: gggg}
\begin{equation*}
\pmb{a}_{n+1}^{(k)}=a_{n+1}a_{0}^{k-1}\binom{k}{1}+\sum_{i=1}^{n}a_{i}\pmb{\widehat{a}}_{n+1-i}^{(k)}
\end{equation*}
Since
\begin{equation*}
\pmb{\widehat{a}}_{n}^{(k)}=\sum_{p=1}^{n}a_{0}^{k-p-1}\binom{k}{p+1}\sum_{\substack{k_1+k_2+\cdots+k_n=p \\ k_1+2k_2+\cdots+nk_n=n}}\frac{p!}{k_{1}!\ldots k_{n}!} a_{1}^{k_{1}}\cdots a_{n}^{k_{n}} 
\end{equation*}
Then we have
\begin{equation*}
\pmb{\widehat{a}}_{n}^{(k)}=\pmb{a}_{n}^{(k-1)}+a_{0}\pmb{\widehat{a}}_{n}^{(k-1)}
\end{equation*}
Applying repeatedly the same procedure, we obtain
\begin{equation*}
\pmb{\widehat{a}}_{n}^{(k)}=\sum_{l=1}^{k-1}a_{0}^{l-1}\pmb{a}_{n}^{(k-l)}+a_{0}^{k-1}\pmb{\widehat{a}}_{n}^{(0)}
\end{equation*}
and since $\pmb{\widehat{a}}_{n}^{(0)}=0$, we immediately have
\begin{equation*}
\pmb{a}_{n+1}^{(k)}=a_{n+1}a_{0}^{k-1}\binom{k}{1}+\sum_{i=1}^{n}a_{i}\sum_{l=1}^{k-1}a_{0}^{l-1}\pmb{a}_{n+1-i}^{(k-l)}
\end{equation*}
as claimed.
\end{proof}
Another interesting result is given in the following theorem
\begin{theorem}\label{thm:w}
Let $\{a_{n}\}_{n\geq 1}$ be a sequence of elements of $R$ and $k$ be an integer. Let
\begin{equation*}
S=1+a_{1}X+a_{2}X^{2}+\cdots=\sum_{i\geq 0}a_{i}X^{i}
\end{equation*}
and
\begin{eqnarray*}
S^{k}=1+\pmb{b}_{1}^{(k)}X+\pmb{b}_{2}^{(k)}X^{2}+\cdots=1+\sum_{i\geq 1}\pmb{b}_{i}^{(k)}X^{i}
\end{eqnarray*}
Then the sequence $\{\pmb{b}_{n}^{(k)}\}_{n\geq 1}$ satisfies the following relations:
\begin{equation*}
\pmb{b}_{n}^{(k)}=\sum_{j=0}^{n}\binom{k}{j}\binom{n-k}{n-j}\pmb{b}_{n}^{(j)}
\end{equation*}
and
\begin{equation*}
\binom{n+k}{n}\pmb{b}_{n}^{(k)}=\sum_{j=0}^{n}\binom{n+k}{n+j}\binom{n+j}{n}\binom{n-k}{n-j}\pmb{b}_{n}^{(j)}
\end{equation*}
\end{theorem}
\begin{proof}
It is clear that
\begin{equation*}
S^{k}=\sum_{i\geq 0}\binom{k}{i}(S-1)^{i}=\sum_{i\geq 0}\binom{k}{i}\sum_{j=0}^{i}\binom{i}{j}(-1)^{i-j}S^{j}
\end{equation*}
If $[X^{n}]S$ denotes the coefficient of $X^{n}$ in the power series $S$. Then it is not difficult to check that
\begin{equation*}
[X^{n}]S^{k}=\sum_{i=0}^{n}\binom{k}{i}\sum_{j=0}^{i}\binom{i}{j}(-1)^{i-j}[X^{n}]S^{j}
\end{equation*}
Thus
\begin{align*}
\pmb{b}_{n}^{(k)}=\sum_{i=0}^{n}\binom{k}{i}\sum_{j=0}^{i}\binom{i}{j}(-1)^{i-j}\pmb{b}_{n}^{(j)}=\sum_{j=0}^{n}\sum_{i=j}^{n}\binom{k}{i}\binom{i}{j}(-1)^{i-j}\pmb{b}_{n}^{(j)}
\end{align*}
Now, we can write
\begin{equation*}
\sum_{i=j}^{n}\binom{k}{i}\binom{i}{j}(-1)^{i-j}=\binom{k}{j}\sum_{i=j}^{n}\binom{k-j}{i-j}(-1)^{i-j}
\end{equation*}
Using some basic tools, we easily see that
\begin{equation*}
\sum_{i=j}^{n}\binom{k-j}{i-j}(-1)^{i-j}=\binom{k-j-1}{n-j}(-1)^{n-j}=\binom{n-k}{n-j}
\end{equation*}
Then
\begin{equation*}
\pmb{b}_{n}^{(k)}=\sum_{j=0}^{n}\binom{k}{j}\binom{n-k}{n-j}\pmb{b}_{n}^{(j)}
\end{equation*}
as claimed. By the first identity, we have
\begin{equation*}
\binom{n+k}{n}\pmb{b}_{n}^{(k)}=\sum_{j=0}^{n}\binom{k}{j}\binom{n+k}{n}\binom{n-k}{n-j}\pmb{b}_{n}^{(j)}
\end{equation*}
Using the relation
\begin{equation*}
\binom{k}{j}\binom{n+k}{n}=\binom{n+j}{j}\binom{n+k}{n+j}
\end{equation*}
we obtain the second desired formula.
\end{proof}
In a commutative unitary ring, it is useful to define the formal derivative of any formal power series. To do this, if we let $S=\sum_{i\geq 0}a_{i}X^{i}$, then the formal derivative of $S$ is defined by $S^{'}=\sum_{i\geq 0}(i+1)a_{i}X^{i}$. For any formal power series $F$ and $G$ the product rule holds, that is
\begin{equation*}
(FG)^{'}=F^{'}G+FG^{'}
\end{equation*}
To prove this statement, let $F=\sum_{i\geq 0}f_{i}X^{i}$ and $G=\sum_{i\geq 0}g_{i}X^{i}$. Then the derivative of the product $FG=\sum_{i\geq 0}\sum_{j\geq 0}f_{i}g_{j}X^{i+j}$ is
\begin{align*}
(FG)^{'}&=\sum_{i\geq 0}\sum_{j\geq 0}(i+j)f_{i}g_{j}X^{i+j-1} \\
        &=\sum_{i\geq 0}\sum_{j\geq 0}if_{i}g_{j}X^{i-1+j}+\sum_{i\geq 0}\sum_{j\geq 0}f_{i}jg_{j}X^{i+j-1} \\
        &=F^{'}G+FG^{'}
\end{align*}
Using the induction, we can easily prove that for all positive integer $k$
\begin{equation*}
(F^{k})^{'}=kF^{'}F^{k-1}
\end{equation*}
If $S=\sum_{i\geq 0}a_{i}X^{i}$ is invertible, then for all integer $k$ it is easy to check that the derivative of $S^{k}$ is
\begin{equation*}
(S^{k})^{'}=kS^{'}S^{k-1}
\end{equation*}
This is a direct consequence of the fact that $S^{k}S^{-k}=1$ and the product rule.\\ 
\indent
The theory of formal power series provides an elegant algebraic framework for several situations. It encodes a great deal of algebraic
manipulations often used to produce stupendously simple proofs of some seemingly difficult problems.
We note that the study of the formal power series in a general unitary commutative ring is quite difficult, especially in proving some relations because we have few standard techniques. In contrast, it should be noted that the formal derivative is an essential tool because we can use it to provide several striking results. For example, the following result can be deduced easily using this useful operation, but difficult to prove it using standard calculus.
\begin{theorem}\label{thm:ww}
Let $\{a_{n}\}_{n\in \mathbb{N}}$ be a sequence of elements of $R$ such that $a_{0}$ is invertible. For all integer $k$, the $k$th power of the following formal power series
\begin{equation*}
S=a_{0}+a_{1}X+a_{2}X^{2}+\cdots=\sum_{i\geq 0}a_{i}X^{i}
\end{equation*}
are
\begin{eqnarray*}
S^{k}=\pmb{a}_{0}^{k}+\pmb{a}_{1}^{(k)}X+\pmb{a}_{2}^{(k)}X^{2}+\cdots=\sum_{i\geq 0}\pmb{a}_{i}^{(k)}X^{i}
\end{eqnarray*}
where the sequence $\{\pmb{a}_{n}^{(k)}\}_{n\geq 1}$ satisfies the following relation:
\begin{equation*}
n\pmb{a}_{n}^{(k)}=k\sum_{i=1}^{n}ia_{i}\pmb{a}_{n-i}^{(k-1)}.
\end{equation*}
\end{theorem}
\begin{proof}
The formal derivative of
\begin{equation*}
S^{k}=\pmb{a}_{0}^{k}+\pmb{a}_{1}^{(k)}X+\pmb{a}_{2}^{(k)}X^{2}+\cdots=\sum_{n\geq 0}\pmb{a}_{n}^{(k)}X^{n}
\end{equation*}
gives
\begin{equation*}
k\bigg(\sum_{n\geq 0}\pmb{a}_{n}^{(k-1)}X^{n}\bigg)\bigg(\sum_{n\geq 0}(n+1)a_{n+1}X^{n}\bigg)=\sum_{n\geq 0}(n+1)\pmb{a}_{n+1}^{(k)}X^{n}
\end{equation*}
This yields
\begin{equation*}
\sum_{n\geq 0}\bigg(\sum_{i=0}^{n}k(i+1)a_{i+1}\pmb{a}_{n-i}^{(k-1)}\bigg)X^{n}=\sum_{n\geq 0}(n+1)\pmb{a}_{n+1}^{(k)}X^{n}
\end{equation*}
Equating the coefficients, we obtain
\begin{equation*}
\sum_{i=0}^{n}k(i+1)a_{i+1}\pmb{a}_{n-i}^{(k-1)}=(n+1)\pmb{a}_{n+1}^{(k)}
\end{equation*}
Then the formula of this theorem follows.
\end{proof}
The following theorem is not new, it is Euler's result introduced at the beginning of this section. Let us prove this elegant result in a general context, which is a commutative ring with unity.
\begin{theorem}\label{thm:www}
Let $\{a_{n}\}_{n\in \mathbb{N}}$ be a sequence of elements of $R$ such that $a_{0}$ is invertible. For all integer $k$, the $k$th power of the following formal power series
\begin{equation*}
S=a_{0}+a_{1}X+a_{2}X^{2}+\cdots=\sum_{i\geq 0}a_{i}X^{i}
\end{equation*}
are
\begin{equation*}
S^{k}=\pmb{a}_{0}^{k}+\pmb{a}_{1}^{(k)}X+\pmb{a}_{2}^{(k)}X^{2}+\cdots=\sum_{i\geq 0}\pmb{a}_{i}^{(k)}X^{i}
\end{equation*}
where the sequence $\{\pmb{a}_{n}^{(k)}\}_{n\geq 0}$ satisfies the following relation:
\begin{equation*}
(n+1)a_{0}\pmb{a}_{n+1}^{(k)}=\sum_{i=0}^{n}(k(i+1)-(n-i))a_{i+1}\pmb{a}_{n-i}^{(k)}, \,\,\ n\geq 0.
\end{equation*}
\end{theorem}
\begin{proof}
We have
\begin{equation*}
S^{k}=\pmb{a}_{0}^{k}+\pmb{a}_{1}^{(k)}X+\pmb{a}_{2}^{(k)}X^{2}+\cdots=\sum_{n\geq 0}\pmb{a}_{n}^{(k)}X^{n}
\end{equation*}
The formal derivative gives
\begin{equation*}
kS^{k-1}S^{'}=\sum_{n\geq 0}(n+1)\pmb{a}_{n+1}^{(k)}X^{n}
\end{equation*}
If we multiply both sides of this formula by $S$, we obtain
\begin{equation*}
kS^{k}S^{'}=\bigg(\sum_{n\geq 0}(n+1)\pmb{a}_{n+1}^{(k)}X^{n}\bigg)\bigg(\sum_{n\geq 0}a_{n}X^{n}\bigg)
\end{equation*}
Thus
\begin{equation*}
kS^{k}S^{'}-\bigg(\sum_{n\geq 0}(n+1)\pmb{a}_{n+1}^{(k)}X^{n}\bigg)\bigg(\sum_{n\geq 1}a_{n}X^{n}\bigg)=a_{0}\sum_{n\geq 0}(n+1)\pmb{a}_{n+1}^{(k)}X^{n}
\end{equation*}
The expansion of this equation is
\begin{equation*}
k\bigg(\sum_{n\geq 0}\pmb{a}_{n}^{(k)}X^{n}\bigg)\bigg(\sum_{n\geq 0}(n+1)a_{n+1}X^{n}\bigg)-\bigg(\sum_{n\geq 0}n\pmb{a}_{n}^{(k)}X^{n}\bigg)\bigg(\sum_{n\geq 0}a_{n+1}X^{n}\bigg)=a_{0}\sum_{n\geq 0}(n+1)\pmb{a}_{n+1}^{(k)}X^{n}
\end{equation*}
Simple manipulation gives
\begin{equation*}
\sum_{n\geq 0}\bigg(\sum_{i=0}^{n}k(i+1)a_{i+1}\pmb{a}_{n-i}^{(k)}-a_{i+1}(n-i)\pmb{a}_{n-i}^{(k)}\bigg)X^{n}=\sum_{n\geq 0}(n+1)a_{0}\pmb{a}_{n+1}^{(k)}X^{n}
\end{equation*}
Therefore
\begin{equation*}
\sum_{n\geq 0}\bigg(\sum_{i=0}^{n}(k(i+1)-(n-i))a_{i+1}\pmb{a}_{n-i}^{(k)}\bigg)X^{n}=\sum_{n\geq 0}(n+1)a_{0}\pmb{a}_{n+1}^{(k)}X^{n}
\end{equation*}
This gives the desired identity.
\end{proof}
\section{Applications to Bernoulli numbers and Stirling numbers of the second kind}
The Bernoulli numbers have numerous interesting applications in combinatorics and number theory. These numbers have been extensively studied and successfully applied, over the last two centuries, in pure and applied mathematics. Therefore, a very large literature has been devoted
to the study of Bernoulli numbers. It is practically impossible to provide a nearly complete bibliography on these well known numbers.\\
The Bernoulli numbers $B_{n}$ can be defined formally by~\cite{Gevere}
\begin{equation*}
\Bigg(1+\sum_{n\geq 1}\frac{1}{(n+1)!}X^{n}\Bigg)^{-1}=1+\sum_{n\geq 1}\frac{1}{n!}B_{n}X^{n}
\end{equation*}
In the vast literature on Bernoulli numbers, there are several results concerning explicit representations of these numbers.
For an excellent readable summary and some extremely interesting historical vignettes, the interested reader is referred to~\cite{Hwgould2}.
In addition, there exists a large number of curious identities involving binomial convolutions of these numbers, the most known being Euler's formula
\begin{equation}\label{eulerr}
\sum_{i=0}^{n}\binom{n}{i}B_{i}B_{n-i}=(1-n)B_{n}-nB_{n-1} \,\ \mbox{for} \,\ n\in\{0,1,2,\ldots \}
\end{equation}
which is equivalent to
\begin{equation*}
\sum_{i=2}^{n-2}\binom{n}{i}B_{i}B_{n-i}=-(n+1)B_{n} \,\ \mbox{for} \,\ n\in\{4,5,6,\ldots \}
\end{equation*}
\indent
In mathematics, the Stirling numbers of the second kind arise in a variety of combinatorics and number theory problems. These useful numbers were introduced by the famous mathematician James Stirling. An excellent treatment of the Stirling numbers of the second kind can be found in the well known research monograph~\cite{Comtet}.\\
The Stirling numbers of the second kind $S(n,k)$ can be defined formally by~\cite{Cacha}
\begin{equation*}
\frac{1}{k!}\Bigg(\sum_{n\geq 1}\frac{1}{n!}X^{n}\Bigg)^{k}=\sum_{n\geq k}\frac{1}{n!}S(n,k)X^{n}
\end{equation*}
In many situations of number theory and calculus, we encounter the Bernoulli numbers. For example, these famous numbers are helpful to express the power sum $1^{m}+2^{m}+\cdots+n^{m}$ as 
\begin{equation*}
1^{m}+2^{m}+\cdots+n^{m}=\frac{1}{m+1}\sum_{j=0}^{m}\binom{m+1}{j}(n+1)^{m-j+1}B_{j}, \,\ m\geq 0, n\geq 0.
\end{equation*}
This identity is the well known Bernoulli formula.
It is important to remark that the power sum can also be calculated in a direct way using the Stirling numbers of the second kind.
\begin{equation*}
1^{m}+2^{m}+\cdots+n^{m}=\sum_{j=0}^{n}\binom{n+1}{j+1}j!S(m,j), \,\ m\geq 0, n\geq 0.
\end{equation*}

The remarkable feature of the Bernoulli numbers is the fact that they appear in the evaluation of the well known Riemann zeta function via the following striking formulas
\begin{equation*}
\zeta(2m)=\sum_{n=1}^{\infty}\frac{1}{n^{2m}}=(-1)^{m-1}\frac{(2\pi)^{2m}}{2(2m)!}B_{^{2m}},
\end{equation*}
and
\begin{equation*}
\zeta(-m)=-\frac{1}{m+1}B_{m+1}.
\end{equation*}
where $m\in\{1,2,3,\ldots \}$.
The intimate relationship between the zeta function and the Bernoulli numbers led to the fact that these numbers have several profound arithmetical properties.\\
\indent
For any integer $m$, the generalized Bernoulli numbers $B_{n}^{(m)}$ can be defined by means of~\cite{Gevere}
\begin{equation}\label{berno}
\Bigg(1+\sum_{n\geq 1}\frac{1}{(n+1)!}X^{n}\Bigg)^{-m}=1+\sum_{n\geq 1}\frac{1}{n!}B_{n}^{(m)}X^{n}
\end{equation}
The generalized Bernoulli numbers are a natural generalization of the classical Bernoulli numbers. They have various interesting properties in mathematics. Also, the generalized Bernoulli numbers have remarkable arithmetic properties see e.g~\cite{Lcarlitz}.
From the corresponding definitions, we easily see that
\begin{equation}\label{bernoo}
\Bigg(1+\sum_{n\geq 1}\frac{1}{n!}B_{n}X^{n}\Bigg)^{m}=1+\sum_{n\geq 1}\frac{1}{n!}B_{n}^{(m)}X^{n}
\end{equation}
\indent
The results of the last section can be used to generate a large variety of elegant formulas and determinant identities for well known and new algebraic identities involving Bernoulli numbers and Stirling numbers of the second kind.\\
\begin{theorem}
For each integer $n\geq 0$, the Bernoulli numbers can be expressed as
\begin{equation*}
B_{n}=n!\sum_{\substack{k_1+2k_2+\cdots+nk_n=n}}\frac{(k_1+k_2+\cdots+k_n)!}{k_1!k_2!\cdots k_n!}\Bigg(\frac{-1}{2!}\Bigg)^{k_1}\cdots \Bigg(\frac{-1}{(n+1)!}\Bigg)^{k_n}
\end{equation*}
Moreover
\begin{equation}\label{deter}
B_{n}=(-1)^{n}n!
\left|\begin{array}{cccccc}
\frac{1}{2!}   & \frac{1}{3!}        &\cdots  &\cdots   & \cdots  &\frac{1}{(n+1)!}\vspace*{0.5pc}\\
 1      &     \ddots    &\ddots  &         &        &\vdots      \vspace*{0.5pc}\\
0       &    \ddots     & \ddots &  \ddots &        &\vdots \vspace*{0.5pc}\\
\vdots  &    \ddots     & \ddots &\ddots   &  \ddots      &\vdots \vspace*{0.5pc}\\
\vdots  &               & \ddots &\ddots   & \ddots & \frac{1}{3!} \vspace*{0.5pc}\\
0       & \cdots        & \cdots &   0     & 1      &\frac{1}{2!}
\end{array}\right|.
\end{equation}
\end{theorem}
\begin{proof}
Follows easily using Theorem~\ref{Thm f1}.
\end{proof}
The expression~\eqref{deter} is not new, it appeared for example in~\cite{Ainselberg,Fqi}.\\
\indent
Another application of Theorem~\ref{Thm f1} is given in the following result.
\begin{theorem}
Let $n$ be a positive integer, we have
\begin{equation*}
\frac{1}{(n+1)!}=\sum_{\substack{k_1+k_2+\cdots+k_n=p \\ k_1+2k_2+\cdots+nk_n=n}}\binom{p}{k_{1},\ldots,k_{n}}(-1)^{p}
\bigg(\frac{1}{1!}B_{1}\bigg)^{k_{1}}\bigg(\frac{1}{2!}B_{2}\bigg)^{k_{2}}\cdots \bigg(\frac{1}{n!}B_{n}\bigg)^{k_{n}} 
\end{equation*}
and
\begin{equation*}
\frac{(-1)^{n}}{(n+1)!}=
\left|\begin{array}{cccccc}
\frac{1}{1!}B_{1}   & \frac{1}{2!}B_{2}        &\cdots  &\cdots   & \cdots  &\frac{1}{n!}B_{n}\vspace*{0.5pc}\\
 1      &     \ddots    &\ddots  &         &        &\vdots      \vspace*{0.5pc}\\
0       &    \ddots     & \ddots &  \ddots &        &\vdots \vspace*{0.5pc}\\
\vdots  &    \ddots     & \ddots &\ddots   &  \ddots      &\vdots \vspace*{0.5pc}\\
\vdots  &               & \ddots &\ddots   & \ddots & \frac{1}{2!}B_{2} \vspace*{0.5pc}\\
0       & \cdots        & \cdots &   0     & 1      &\frac{1}{1!}B_{1}
\end{array}\right|.
\end{equation*}
\end{theorem}
\begin{proof}
By definition, we have
\begin{equation*}
\Bigg(1+\sum_{n\geq 1}\frac{1}{n!}B_{n}X^{n}\bigg)^{-1}=1+\sum_{n\geq 1}\frac{1}{(n+1)!}X^{n}
\end{equation*}
Taking into account Theorem~\ref{Thm f1}, the identity of this theorem and its determinant expression follows.
\end{proof}

\begin{theorem}
Let $m$ be an integer and $n$ be a nonnegative integers. Then
\begin{equation*}
\frac{1}{n!} B_{n}^{(m)}=\sum_{\substack{k_1+k_2+\cdots+k_n=l \\ k_1+2k_2+\cdots+nk_n=n}}(-1)^{l}\binom{m+l-1}{l}\frac{l!}{k_1!k_2!\cdots k_n!}\prod_{i=1}^{n}
\Bigg(\frac{1}{(i+1)!}\Bigg)^{k_{i}}
\end{equation*}
\end{theorem}
\begin{proof}
Using Formula~\eqref{berno} and a direct application of Theorem~\eqref{thm1} gives
\begin{equation*}
\frac{1}{n!} B_{n}^{(m)}=\sum_{\substack{k_1+k_2+\cdots+k_n=l \\ k_1+2k_2+\cdots+nk_n=n}}\binom{-m}{l}l!\frac{1}{k_1!k_2!\cdots k_n!}\prod_{i=1}^{n}
\Bigg(\frac{1}{(i+1)!}\Bigg)^{k_{i}}
\end{equation*}
Since $\binom{-m}{l}(-1)^{l}=\binom{m+l-1}{l}$, then we obtain the result.
\end{proof}
\begin{theorem}
Let $m$ be an integer and $n$ be a nonnegative integer, then
\begin{equation*}
\frac{1}{n!}B_{n}^{(m)}=\sum_{\substack{k_1+k_2+\cdots+k_n=l \\ k_1+2k_2+\cdots+nk_n=n}}\binom{m}{l}\frac{l!}{k_1!k_2!\cdots k_n!}\prod_{i=1}^{n}\bigg(\frac{1}{i!}B_{i}\bigg)^{k_{i}}
\end{equation*}
\end{theorem}
\begin{proof}
Formula~\eqref{berno} tells us that
\begin{equation*}
\Bigg(1+\sum_{n\geq 1}\frac{1}{n!}B_{n}X^{n}\Bigg)^{m}=1+\sum_{n\geq 1}\frac{1}{n!}B_{n}^{(m)}X^{n}
\end{equation*}
and the result of Theorem~\eqref{thm1} gives
\begin{equation*}
\frac{1}{n!}B_{n}^{(m)}=\sum_{\substack{k_1+k_2+\cdots+k_n=l \\ k_1+2k_2+\cdots+nk_n=n}}\binom{m}{l}l!\frac{1}{k_1!k_2!\cdots k_n!}\prod_{i=1}^{n}\bigg(\frac{1}{i!}B_{i}\bigg)^{k_{i}}
\end{equation*}
and so the result is shown.
\end{proof}
\begin{theorem}
Let $n,m$ be two positive integers, then
\begin{equation*}
\frac{(-1)^{n}}{n!}B_{n}^{(m)}=
\left|\begin{array}{cccccc}
\frac{1}{1!}B_{1}^{(-m)}   & \frac{1}{2!}B_{2}^{(-m)}        &\cdots  &\cdots   & \cdots  &\frac{1}{n!}B_{n}^{(-m)}\vspace*{0.5pc}\\
 1      &     \ddots    &\ddots  &         &        &\vdots      \vspace*{0.5pc}\\
0       &    \ddots     & \ddots &  \ddots &        &\vdots \vspace*{0.5pc}\\
\vdots  &    \ddots     & \ddots &\ddots   &  \ddots      &\vdots \vspace*{0.5pc}\\
\vdots  &               & \ddots &\ddots   & \ddots & \frac{1}{2!}B_{2}^{(-m)} \vspace*{0.5pc}\\
0       & \cdots        & \cdots &   0     & 1      &\frac{1}{1!}B_{1}^{(-m)}
\end{array}\right|.
\end{equation*}

where
\begin{equation*}
\frac{1}{n!}B_{n}^{(m)}=\sum_{\substack{k_1+k_2+\cdots+k_n=l \\ k_1+2k_2+\cdots+nk_n=n}}(-1)^{l}\frac{l!}{k_1!k_2!\cdots k_n!}
\bigg(\frac{1}{1!}B_{1}^{(-m)}\bigg)^{k_1}\bigg(\frac{1}{2!}B_{2}^{(-m)}\bigg)^{k_2}\cdots \bigg(\frac{1}{n!}B_{n}^{(-m)}\bigg)^{k_n}
\end{equation*}
\end{theorem}
\begin{proof}
By Formula~\eqref{berno}, it is clear that
\begin{equation*}
\Bigg(1+\sum_{n\geq 1}\frac{1}{(n+1)!}X^{n}\Bigg)^{m}=1+\sum_{n\geq 1}\frac{1}{n!}B_{n}^{(-m)}X^{n}
\end{equation*}
Then we obtain the following identity
\begin{equation*}
\Bigg(1+\sum_{n\geq 1}\frac{1}{(n+1)!}X^{n}\Bigg)^{-m}=1+\sum_{n\geq 1}\frac{1}{n!}B_{n}^{(m)}X^{n}
\end{equation*}
this identity entails 
\begin{equation*}
\Bigg(1+\sum_{n\geq 1}\frac{1}{n!}B_{n}^{(-m)}X^{n}\Bigg)^{-1}=1+\sum_{n\geq 1}\frac{1}{n!}B_{n}^{(m)}X^{n}
\end{equation*}
Therefore the result follows directly from Theorem~\ref{Thm f1}.
\end{proof}

\begin{theorem}
Let $n$ and $k$ be two positive integer such that $n\geq k\geq 1$. We have the following well known identity 
\begin{equation*}
S(n,k)=\sum_{\substack{k_1+k_2+\cdots+k_n=k \\ k_1+2k_2+\cdots+nk_n=n}}\frac{n!}{k_1!k_2!\cdots k_n!}\Bigg(\frac{1}{1!}\Bigg)^{k_{1}}\Bigg(\frac{1}{2!}\Bigg)^{k_2}\cdots \Bigg(\frac{1}{n!}\Bigg)^{k_n}
\end{equation*}
Accordingly, one has
\begin{equation*}
S(n+1,k+1)=\frac{1}{k+1}\sum_{i=1}^{n+1-k}\binom{n+1}{i}S(n+1-i,k)
\end{equation*}
\end{theorem}
\begin{proof}
Using Theorem~\ref{thm1}, we have
\begin{equation*}
\Bigg(\sum_{n\geq 1}\frac{1}{n!}X^{n}\Bigg)^{k}=\sum_{n\geq 0}[\mathcal{X}_{n}]_{k}^{n}(0)X^{n}
\end{equation*}
On the other hand, it is clear that
\begin{equation*}
[\mathcal{X}_{n}]_{k}^{n}(0)=\sum_{\substack{k_1+k_2+\cdots+k_n=k \\ k_1+2k_2+\cdots+nk_n=n}}\frac{k!}{k_1!k_2!\cdots k_n!}\Bigg(\frac{1}{1!}\Bigg)^{k_{1}}\Bigg(\frac{1}{2!}\Bigg)^{k_2}\cdots \Bigg(\frac{1}{n!}\Bigg)^{k_n}
\end{equation*}
Note that if $n<k$, then $[\mathcal{X}_{n}]_{k}^{n}(0)=0$. Consequently, equating the coefficients of $x^{n}$, we obtain
\begin{equation*}
\frac{1}{n!}S(n,k)=\sum_{\substack{k_1+k_2+\cdots+k_n=k \\ k_1+2k_2+\cdots+nk_n=n}}\frac{1}{k_1!k_2!\cdots k_n!}\Bigg(\frac{1}{1!}\Bigg)^{k_{1}}\Bigg(\frac{1}{2!}\Bigg)^{k_2}\cdots \Bigg(\frac{1}{n!}\Bigg)^{k_n}
\end{equation*}
Therefore, we get the first identity. The second identity follows directly from Lemma $2.2$ of~\cite{Mou}.
\end{proof}
One of the most appealing aspects of the Stirling numbers of the second kind is their connection with the generalized Bernoulli numbers
\begin{theorem}\label{thm: kkk}
For all $k\geq 0$ and $n\geq 1$
\begin{equation*}
\sum_{i=0}^{n}\binom{n+k}{n-i}B_{n-i}^{(k)}S(i+k,k)=0
\end{equation*}
In particular
\begin{equation*}
\sum_{i=0}^{n}\binom{n+1}{n-i}B_{n-i}=0
\end{equation*}
\end{theorem}
\begin{proof}
By definition, we have
\begin{equation*}
\frac{1}{k!}X^{k}\Bigg(1+\sum_{n\geq 1}\frac{1}{(n+1)!}X^{n}\Bigg)^{k}=\sum_{n\geq k}\frac{1}{n!}S(n,k)X^{n}
\end{equation*}
It is not difficult to check
\begin{equation*}
\Bigg\{1+\sum_{n\geq 1}\frac{1}{n!}B_{n}^{(k)}X^{n}\Bigg\}\Bigg\{\sum_{n\geq 0}\frac{1}{(n+k)!}S(n+k,k)X^{n}\Bigg\}=\frac{1}{k!}
\end{equation*}
Therefore
\begin{equation*}
\sum_{n\geq 0}\Bigg(\sum_{i=0}^{n}\frac{1}{(n-i)!}B_{n-i}^{(k)}\frac{1}{(i+k)!}S(i+k,k)\Bigg)X^{n}=\frac{1}{k!}
\end{equation*}
This completes the proof.
\end{proof}

\begin{theorem}
For all $k\geq 0$ and $n\geq 1$
\begin{equation*}
\binom{n+k}{k}B_{n}^{(-k)}=S(n+k,k)
\end{equation*}
and
\begin{equation*}
\sum_{i=0}^{n}\binom{n+k}{n-i}\binom{i+k}{i}B_{n-i}^{(k)}B_{i}^{(-k)}=0
\end{equation*}
\end{theorem}
\begin{proof}
Formula~\eqref{berno} tells us that for any integer $m$, the generalized Bernoulli numbers can be defined formally by
\begin{equation*}
\Bigg(1+\sum_{n\geq 1}\frac{1}{(n+1)!}X^{n}\Bigg)^{m}=\sum_{n\geq 0}\frac{1}{n!}B_{n}^{(-m)}X^{n}
\end{equation*}
On the other hand, we have
\begin{equation*}
\frac{1}{k!}X^{k}\Bigg(1+\sum_{n\geq 1}\frac{1}{(n+1)!}X^{n}\Bigg)^{k}=\sum_{n\geq k}\frac{1}{n!}S(n,k)X^{n}
\end{equation*}
Therefore
\begin{equation*}
\sum_{n\geq 0}\frac{1}{k!n!}B_{n}^{(-k)}X^{n+k}=\sum_{n\geq 0}\frac{1}{(n+k)!}S(n+k,k)X^{n+k}
\end{equation*}
This completes the proof of the first statement. The second statement is a consequence of Theorem\eqref{thm: kkk}
\end{proof}
\begin{theorem}
For all $k\geq 1$ and $n\geq 1$, we have
\begin{equation*}
B_{n}^{(k)}=\binom{k}{1}B_{n}+\sum_{i=1}^{n-1}\binom{n}{i}B_{i}\sum_{l=1}^{k-1}B_{n-i}^{(k-l)}
\end{equation*}
and
\begin{equation*}
B_{n}^{(-k)}=\binom{k}{1}\frac{1}{n+1}+\sum_{i=1}^{n-1}\frac{1}{i+1}\binom{n}{i}\sum_{l=1}^{k-1}B_{n-i}^{(-k-l)}
\end{equation*}
\end{theorem}
\begin{proof}
The first identity follows easily from Theorem\eqref{thm:bern} and Formula~\eqref{bernoo}.
To prove the second identity, we first applied Theorem\eqref{thm:bern} to Formula~\eqref{berno} and we obtain
\begin{equation*}
\frac{1}{n!}B_{n}^{(-k)}=\binom{k}{1}\frac{1}{(n+1)!}+\sum_{i=1}^{n-1}\frac{1}{(i+1)!}\sum_{l=1}^{k-1}\frac{1}{(n-i)!}B_{n-i}^{(-k-l)}
\end{equation*}
which implies
\begin{equation*}
B_{n}^{(-k)}=\binom{k}{1}\frac{1}{n+1}+\sum_{i=1}^{n-1}\frac{1}{(i+1)i!}\frac{n!}{(n-i)!}\sum_{l=1}^{k-1}B_{n-i}^{(-k-l)}
\end{equation*}
Therefore
\begin{equation*}
B_{n}^{(-k)}=\binom{k}{1}\frac{1}{n+1}+\sum_{i=1}^{n-1}\frac{1}{i+1}\binom{n}{i}\sum_{l=1}^{k-1}B_{n-i}^{(-k-l)}
\end{equation*}
By the above, this theorem is proved.
\end{proof}
The following result is an immediate consequence of Theorem~\eqref{thm:w}.
\begin{theorem}
Let $k$ be an integer. Then the generalized Bernoulli numbers satisfy the following two formulas
\begin{equation*}
B_{n}^{(k)}=\sum_{j=0}^{n}\binom{k}{j}\binom{n-k}{n-j}B_{n}^{(j)}, \,\ n\geq 0,
\end{equation*}
and
\begin{equation*}
\binom{n+k}{n}B_{n}^{(k)}=\sum_{j=0}^{n}\binom{n+k}{n+j}\binom{n+j}{n}\binom{n-k}{n-j}B_{n}^{(j)}, \,\ n\geq 0.
\end{equation*}
\end{theorem}

\begin{theorem}
Let $k$ be an integer. Then the generalized Bernoulli numbers satisfy the following
\begin{equation*}
B_{n+1}^{(k)}=k\sum_{i=0}^{n}\binom{n}{i}B_{i+1}B_{n-i}^{(k-1)}.
\end{equation*}
and
\begin{equation*}
B_{n+1}^{(-k)}=k\sum_{i=0}^{n}\binom{n}{i}\frac{1}{i+2}B_{n-i}^{(-k-1)}.
\end{equation*}
\end{theorem}
\begin{proof}
This result is a direct application Theorem~\eqref{thm:ww} and the following relations
\begin{equation*}
\Bigg(1+\sum_{n\geq 1}\frac{1}{n!}B_{n}X^{n}\Bigg)^{k}=1+\sum_{n\geq 1}\frac{1}{n!}B_{n}^{(k)}X^{n}
\end{equation*}
and
\begin{equation*}
\Bigg(1+\sum_{n\geq 1}\frac{1}{(n+1)!}X^{n}\Bigg)^{k}=1+\sum_{n\geq 1}\frac{1}{n!}B_{n}^{(-k)}X^{n}
\end{equation*}
\end{proof}
\begin{theorem}
Let $k$ be an integer. Then the generalized Bernoulli numbers satisfy the following
\begin{equation*}
B_{n+1}^{(k)}=\sum_{i=0}^{n}\bigg(k\binom{n}{i}-\binom{n}{i+1}\bigg)B_{i+1}B_{n-i}^{(k)}.
\end{equation*}

\begin{equation*}
B_{n+1}^{(-k)}=\sum_{i=0}^{n}\bigg(k\binom{n}{i}-\binom{n}{i+1}\bigg)\frac{1}{i+2}B_{n-i}^{(-k)}.
\end{equation*}
\end{theorem}
\begin{proof}

Taking into account Theorem~\eqref{thm:www} and Formula~\eqref{bernoo}, we have the following relation
\begin{equation*}
\frac{1}{n!}B_{n+1}^{(k)}=\sum_{i=0}^{n}(k(i+1)-(n-i))\frac{1}{(i+1)!}B_{i+1}\frac{1}{(n-i)!}B_{n-i}^{(k)}.
\end{equation*}
which implies
\begin{equation*}
B_{n+1}^{(k)}=\sum_{i=0}^{n}\bigg(\frac{kn!(i+1)}{(i+1)!(n-i)!}-\frac{n!(n-i)}{(i+1)!(n-i)!}\bigg)B_{i+1}B_{n-i}^{(k)}.
\end{equation*}
and
\begin{equation*}
B_{n+1}^{(k)}=\sum_{i=0}^{n}\bigg(\frac{kn!}{i!(n-i)!}-\frac{n!}{(i+1)!(n-i-1)!}\bigg)B_{i+1}B_{n-i}^{(k)}.
\end{equation*}
this is just the first statement. The proof of the second statement can be verified by proceeding as in the proof of the first identity with the aid of Formula~\eqref{berno}.
\end{proof} 
In the following result, we provide a more interesting and potentially useful recursive formula for the generalized Bernoulli numbers.
This recurrence formula allows us to compute these numbers efficiently. Also, this recursive expression can be used to generate interesting useful identities. For example, we use it to provide a generalization of a known identity due to Euler.
\begin{theorem}\label{Thm ooo}
Let $k$ be an integer, not equal to $0$. We have for all positive integer $n\geq1$
\begin{equation*}
B_{n}^{(k+1)}=\bigg(1-\frac{n}{k}\bigg)B_{n}^{(k)}-nB_{n-1}^{(k)}.
\end{equation*}
\end{theorem}
\begin{proof}
Let
\begin{equation*}
b_{k}(X)=\Bigg(1+\sum_{n\geq 1}\frac{1}{(n+1)!}X^{n}\Bigg)^{-k}
\end{equation*}
Taking the derivative, we obtain
\begin{equation*}
b_{k}^{'}(X)=-k\Bigg(1+\sum_{n\geq 1}\frac{1}{(n+1)!}X^{n}\Bigg)^{-k-1}\Bigg(\sum_{n\geq 0}\frac{n+1}{(n+2)!}X^{n}\Bigg)
\end{equation*}
Thus
\begin{equation*}
Xb_{k}^{'}(X)=-kb_{k+1}(X)\Bigg(\sum_{n\geq 1}(\frac{1}{n!}-\frac{1}{(n+1)!})X^{n}\Bigg)
\end{equation*}
and hence
\begin{align*}
Xb_{k}^{'}(X)&=-kXb_{k+1}(X)\Bigg(\sum_{n\geq 1}\frac{1}{n!}X^{n-1}\Bigg)+kb_{k+1}(X)\Bigg(\sum_{n\geq 1}\frac{1}{(n+1)!}X^{n}\Bigg)\\
             &=-kXb_{k+1}(X)\Bigg(1+\sum_{n\geq 1}\frac{1}{(n+1)!}X^{n}\Bigg)+kb_{k+1}(X)\Bigg(-1+1\sum_{n\geq 1}\frac{1}{(n+1)!}X^{n}\Bigg)\\
             &=-kXb_{k}(X)-kb_{k+1}(X)+kb_{k}(X)\\
\end{align*}
Therefore
\begin{equation*}
kb_{k+1}(X)=k(1-X)b_{k}(X)-Xb_{k}^{'}(X)
\end{equation*}
We know that
\begin{equation*}
b_{k}(X)=1+\sum_{n\geq 1}\frac{1}{n!}B_{n}^{(k)}X^{n}
\end{equation*}
It follows that
\begin{equation*}
k\Bigg(1+\sum_{n\geq 1}\frac{1}{n!}B_{n}^{(k+1)}X^{n}\Bigg)=k(1-X)\Bigg(1+\sum_{n\geq 1}\frac{1}{n!}B_{n}^{(k)}X^{n}\Bigg)-X\Bigg(\sum_{n\geq 0}\frac{1}{n!}B_{n+1}^{(k)}X^{n}\Bigg)
\end{equation*}
By comparing the coefficients, we easily obtain 
\begin{equation*}
kB_{n}^{(k+1)}=kB_{n}^{(k)}-knB_{n-1}^{(k)}-nB_{n}^{(k)}  \,\ \mbox{for} \,\ n\geq 1
\end{equation*}
and
\begin{equation*}
B_{n}^{(k+1)}=\bigg(1-\frac{n}{k}\bigg)B_{n}^{(k)}-nB_{n-1}^{(k)} \,\ \mbox{for} \,\ n\geq 1.
\end{equation*}
as claimed.
\end{proof}
It may be of great interest to provide a natural generalization of Euler's identity~\eqref{eulerr} cited at the beginning of this section. The purpose of the following result is such generalization identity.
\begin{theorem}
Let $k$ be an integer, not equal to $0$, and $n$ a nonnegative integer. The generalized Bernoulli numbers and the Bernoulli numbers can be related by the following relation
\begin{equation*}
\sum_{i=0}^{n}\binom{n}{i}B_{i}B_{n-i}^{(k)}=\bigg(1-\frac{n}{k}\bigg)B_{n}^{(k)}-nB_{n-1}^{(k)}, \,\ n\geq0
\end{equation*}
or equivalently 
\begin{equation*}
\sum_{i=2}^{n-2}\binom{n}{i}B_{i}B_{n-i}^{(k)}=-\frac{n}{k}B_{n}^{(k)}-\frac{1}{2}nB_{n-1}^{(k)}+\frac{1}{2}knB_{n-1}-B_{n}, \,\ n\geq 4.
\end{equation*}
\end{theorem}
\begin{proof}
We start this proof by writing the following obvious identity
\begin{equation*}
\Bigg(1+\sum_{n\geq 1}\frac{1}{n!}B_{n}^{(k+1)}X^{n}\Bigg)=\Bigg(1+\sum_{n\geq 1}\frac{1}{n!}B_{n}X^{n}\Bigg)\Bigg(\sum_{n\geq 0}\frac{1}{n!}B_{n}^{(k)}X^{n}\Bigg)
\end{equation*}
from this, we can derive
\begin{equation*}
B_{n}^{(k+1)}=\sum_{i=0}^{n}\binom{n}{i}B_{i}B_{n-i}^{(k)}
\end{equation*}
On the other hand by Theorem~\eqref{Thm ooo}, we get
\begin{equation*}
\sum_{i=0}^{n}\binom{n}{i}B_{i}B_{n-i}^{(k)}=\bigg(1-\frac{n}{k}\bigg)B_{n}^{(k)}-nB_{n-1}^{(k)}
\end{equation*}
Therefore we reach the desired formula.
\end{proof}

\end{document}